\documentclass[lettersize,journal]{IEEEtran}

\usepackage{amsmath,amsfonts}
\usepackage{algorithm}
\usepackage{algpseudocode}
\usepackage{array}
\usepackage[caption=false,font=normalsize,labelfont=sf,textfont=sf]{subfig}
\usepackage{textcomp}
\usepackage{stfloats}
\usepackage{url}
\usepackage{verbatim}
\usepackage{graphicx}
\usepackage{cite}
\usepackage{mathtools}
\usepackage{amsthm}
\usepackage{orcidlink}

\newtheorem{theorem}{Theorem}
\newtheorem{proposition}{Proposition}
\newtheorem{lemma}{Lemma}

\newtheorem{assumption}{Assumption}

\theoremstyle{definition}
\newtheorem{definition}{Definition}

\newcommand{\TP}{{d}}
\newcommand{\NF}{{n}}

\hypersetup{
    colorlinks = true,
    allcolors = {blue}
}

\begin{document}

\title{Alleviating the Curse of Dimensionality in Minkowski Sum Approximations of Storage Flexibility}

\author{Emrah Öztürk \orcidlink{0000-0001-5063-531X}, Timm Faulwasser \orcidlink{0000-0002-6892-7406}, \IEEEmembership{Senior Member, IEEE}, Karl Worthmann \orcidlink{0000-0002-1450-2373}, Markus Preißinger \orcidlink{0000-0001-5701-9595}, \\ and Klaus Rheinberger \orcidlink{0000-0002-8277-4944}

\thanks{EÖ, KR and MP: The financial support by the Austrian Federal Ministry for Digital and Economic Affairs and the National Foundation for Research, Technology and Development and the Christian Doppler Research Association is gratefully acknowledged. The project was also funded by the federal state of Vorarlberg and the “European Regional Development Fund"}
\thanks{KW gratefully acknowledges funding by the Deutsche Forschungsgemeinschaft (DFG, German Research Foundation) – Project-ID 507037103}
\thanks{EÖ and KR are with the Research Center Energy, Vorarlberg University of Applied Sciences, 6850 Dornbirn, Austria (e-mail: emrah.oeztuerk@fhv.at; klaus.rheinberger@fhv.at).}
\thanks{TF is with the Institute of Energy Systems, Energy
Efficiency and Energy Economics, TU Dortmund University, 44227 Dortmund, Germany (e-mail: timm.faulwasser@ieee.org).}
\thanks{KW is with the Institute of Mathematics,
Technische Universität Ilmenau, 98693 Ilmenau, Germany (e-mail:
karl.worthmann@tu-ilmenau.de).}
\thanks{MP is with the Josef Ressel Centre for Intelligent Thermal Energy Systems, illwerke vkw Endowed Professorship for Energy Efficiency, Research Center Energy, Vorarlberg University of Applied Sciences, 6850 Dornbirn, Austria (e-mail: markus.preissinger@fhv.at).}}

\maketitle

\begin{abstract}
Many real-world applications require the joint optimization of a large number of flexible devices over time.
The flexibility of, e.g., multiple batteries, thermostatically controlled loads, or electric vehicles can be used to support grid operation and to reduce operation costs. 
Using piecewise constant power values, the flexibility of each device over $d$ time periods can be described as a polytopic subset in power space. 
The aggregated flexibility is given by the Minkowski sum of these polytopes. 
As the computation of Minkowski sums is in general demanding, several approximations have been proposed in the literature. 
Yet, their application potential is often objective-dependent and limited by the curse of dimensionality. 
We show that up to $2^d$ vertices of each polytope can be computed efficiently and that the convex hull of their sums provides a computationally efficient inner approximation of the Minkowski sum. 
Via an extensive simulation study, we illustrate that our approach outperforms ten state-of-the-art inner approximations in terms of computational complexity and accuracy for different objectives. 
Moreover, we propose an efficient disaggregation method applicable to any vertex-based approximation. 
The proposed methods provide an efficient means to aggregate and to disaggregate 
energy storages in quarter-hourly periods over an entire day with reasonable accuracy for aggregated cost and for peak power optimization.
\end{abstract}

\begin{IEEEkeywords}
distributed energy resources, energy storage, flexibility aggregation, Minkowski sum, vertex-based approximation, ancillary services, demand response, energy communities
\end{IEEEkeywords}

\section{Introduction}\label{sec:introduction}
\IEEEPARstart{T}{he} coordinated control of a large number of distributed flexible devices offers significant potential for power grids.
For example, the flexibility of shiftable loads in the distribution grid, such as batteries, refrigerators, heat pumps, water heaters, and air conditioners, can be used to support grid operations and to reduce operation costs. 
Eventually, for the sake of computational tractability, the large number of devices necessitates to cluster units and their flexibilities. 
To this end, the concept of an aggregator is introduced in the literature, cf.~\cite{b20}. 
The aggregator is typically an entity located between consumers, energy markets, and network operators. This entity manages contracted consumer devices, estimates the collective flexibility, and assigns power profiles to individual devices. 
The aggregator thus serves as an interface to a virtual power plant, see also~\cite{WortKell14}. 
The flexibility of each device can be described by a subset in the power space and the aggregated flexibility by the point-wise sum of these sets. 
However, the computation of this Minkowski sum is often prohibitive, cf.~\cite{b1}. 
Therefore, various tailored approximations have been proposed in the literature.

Existing approximations can be roughly divided into top-down and bottom-up approaches. 
The former typically use machine learning, Markov chains, etc. to directly approximate the aggregated flexibility, cf.~\cite{b2,b3,b4}. 
The latter start from individual flexibilities, they usually assume a certain underlying structure, and they can be further divided into inner and outer approximations. 
Outer approximations \cite{b5,b6,b7,b8,b9,b10}, compute supersets of the Minkowski sum and therefore they have the major drawback to potentially contain infeasible elements. 
Inner approximations make up the majority of Minkowski sum approximations in the literature \cite{b8,b9,b11,b12,b13,b7,b14,b5,b15,b16,b25,b26,b27}.
However, many of these have drawbacks, such as poorer optimization results compared to a setting without flexibility, high computational burden, and objective-dependent performance, cf.~\cite{b17}. 
Indeed the computational burden limits the application potential of several approaches significantly.
The objective-dependent performance is likely induced by the employed underlying set parametrizations, e.g., an ellipsoid inscribed in a polytope covers the interior rather than the vertices, resulting in poor performance in cost optimization and in good performance for peak reduction. 
An attempt to avoid the underlying structure is made in \cite{b18}, where a recursive algorithm is proposed to compute the vertices of a polytope by computing extreme bounds. 
Yet, this approach suffers from combinatorial complexity as it attempts to compute all vertices with a scheme that may lead to redundant computations. 
However, a related idea will also be used for the method proposed in the present paper. 
Further aggregation strategies, such as characterizing the flexibility of a fleet of heterogeneous storage units using the so-called E-p transform, can be found in \cite{b29,b30,b31}; strategies in the presence of nonlinearities with probabilistic inputs are discussed by \cite{b32,b33,b34}. 
There also exists a dynamic programming approach \cite{b35} and an exact aggregation strategy for a population of electric vehicles using permutahedra \cite{b36}.

Disaggregation represents the inverse operation to aggregation, i.e., the distribution of power profiles across individual flexible devices, cf.~\cite{b12,b5,b25}. 
Existing methods are often based on the solution of optimization problems which may induce a significant computational burden.

The novelty of the present paper is threefold.
First, we propose an efficient vertex-based inner approximation for typical energy storages that overcomes the weaknesses of existing approximations.
Second, the proposed approximation method is benchmarked against ten state-of-the-art inner approximation techniques from the literature. 
It is shown to outperform the other methods in terms of accuracy for various objectives and in terms of computational performance.
Finally, we propose an efficient disaggregation method that does not require optimization and that can be combined with any vertex-based approximation. 

The remainder of this paper is organized as follows: Definitions are given in Section \ref{sec:2}, where we define our approach for all polytopes satisfying two assumptions, discuss its properties, and give example polytopes for illustration purposes.
Section \ref{sec:3} discusses the general results related to our approach.
In Section \ref{sec:4} we propose an efficient algorithm to compute the approximation for energy storages with unrestricted final energy and extend it to the case of restricted final energy by applying corrections.
In Section \ref{sec:5}, we test our approximation against 10 state-of-the-art inner approximations in terms of accuracy for various objectives and computational complexity.
Section \ref{sec:6} is devoted to the novel disaggregation method that applies to all vertex-based approximations. Finally, conclusions are drawn in Section \ref{sec:7}.\\

\textbf{Notation}: The sets of natural and real numbers are denoted by $\mathbb{N}= \{1, 2, \ldots \}$ and $\mathbb{R}$, respectively. 
The Minkowski sum of sets $\mathcal{X}_i \subseteq \mathbb{R}^\TP$, $i \in \{1, \ldots, \NF\}$ is defined by $\mathcal{M} \coloneqq \{x \in \mathbb{R}^\TP : x = \sum_{i=1}^{\NF}x_i,\ x_i \in \mathcal{X}_i\}$.
For a matrix $A \in \mathbb{R}^{k \times \TP}$ and a vector $b \in \mathbb{R}^k$, the set $\mathcal{P}(A, b) \coloneqq \{x \in \mathbb{R}^\TP: Ax \leq b\}$ is a polyhedron, and a polytope if it is bounded.
The convex hull of a set $\mathcal{X}$ is written as $\text{Conv}(\mathcal{X})$.
The $\TP$-dimensional vector of zeros and ones are written as $\mathbf{0}_\TP$ and $\mathbf{1}_\TP$, respectively.
For $x \in \mathbb{R}^\TP$ and $t \leq \TP$, we use the notation $\textnormal{Proj}^t(x) \coloneqq (x_1, \ldots, x_t, \mathbf{0}_{\TP-t})^\top$ for the projection of~$x$ onto its first $t$~components.
The vector consisting of the first $t$ components of a vector $x \in \mathbb{R}^\TP$ is denoted by $x_{[t]} \in \mathbb{R}^t$. A matrix with constant diagonals descending from left to right is called a Toeplitz matrix. We say $v \in \mathbb{R}^\TP$ is a proper convex combination of $p, q \in \mathbb{R}^\TP$ if $v = tp + (1-t)q$, with $p \neq q$ and $t \in (0, 1)$. The cardinality of a set $\mathcal{X}$ is denoted by $\lvert \mathcal{X} \rvert$.

\section{Preliminaries}\label{sec:2}
In this section, we introduce our assumptions and give example polytopes to illustrate the imposed assumptions. Further, we define vectors of extreme actions within these polytopes.
We consider the following assumptions for $\mathcal{P}(A, b) \subset \mathbb{R}^d$.
\begin{assumption}[Required flexibility]
    If $\textnormal{Proj}^t(x) \in \mathcal{P}(A, b)$ for $t \in \{1,\ldots,\TP-1\}$, then there exists an $\varepsilon \in \mathbb{R} \setminus \{0\}$ such that $(x_1, \ldots, x_t, \varepsilon, \mathbf{0}_{\TP-(t+1)})^\top \in \mathcal{P}(A, b)$. Furthermore, there exists an $\varepsilon \in \mathbb{R} \setminus \{0\}$ such that $(\varepsilon, \mathbf{0}_{\TP-1})^\top \in \mathcal{P}(A, b)$.\label{assumption}
\end{assumption}
\begin{assumption}[Projection feasibility]\label{assumption2}
    If $x \in \mathcal{P}(A, b)$, then $\textnormal{Proj}^t(x) \in \mathcal{P}(A, b)$ for all $t \in \{1, \ldots, \TP-1\}$. Furthermore, $\mathbf{0}_{\TP} \in \mathcal{P}(A, b)$.
\end{assumption}
Assumption \ref{assumption} requires a minimum flexibility in each time period, and Assumption \ref{assumption2} requires the feasibility of all projections of $x$ if $x$ is feasible, cf. Fig.~\ref{fig:allowednotallowed}.
The inclusion of the zero vector models not using the flexibility.
\begin{figure}[tb]
	\centering
	\centerline{\includegraphics[width=\columnwidth]{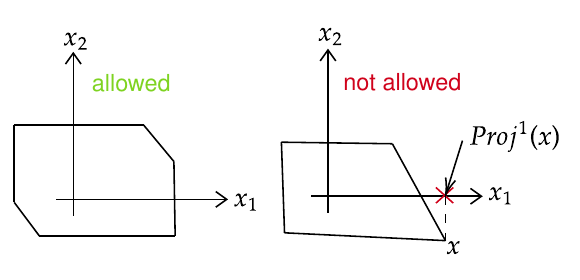}}
	\caption{Illustration of Assumption~\ref{assumption2}. While the polytope on the left satisfies the assumption, the polytope on the right does not.}
	\label{fig:allowednotallowed}
\end{figure}

In the following, we characterize polytopes that are typically used to model the flexibility of energy storages.
These polytopes are parameterized by the vector $p = (\alpha, \underline{x}, \overline{x}, \underline{S}, \overline{S}, \Delta t)^\top \in (0, 1] 
\times \mathbb{R}^4 \times (0,\infty)$ which denotes, respectively, the self-discharge factor, lower and upper bound on the charging rate (kW), minimum and maximum State of Charge (SoC; kWh), and time step (h). Furthermore, the number of time periods is denoted by $\TP \in \mathbb{N}$, initial SoC by $S_0 \in [\underline{S}, \overline{S}]$ (kWh), and minimum final SoC by $S_\text{f} \in [\underline{S}, \overline{S}]$ (kWh).
The set of feasible power profiles $x\in \mathbb{R}^{\TP}$ is given by the system dynamics
    \begin{subequations} 
	    \begin{align}
	        &\underline{x} \leq x(t) \leq \overline{x} \quad \forall \ t = 1, \ldots, d\\
	        &S(t) = \alpha S(t-1) + x(t)\Delta t \quad \forall \ t = 1, \ldots, d\\
	        &\underline{S} \leq S(t) \leq \overline{S} \quad \forall \ t = 1, \ldots, d-1\\
	        &S(0) = S_0\\
	        &S_f \leq S(d) \leq \overline{S}
	    \end{align}
	\end{subequations}
    and results in the polytope
    \begin{equation}
        \mathcal{B}(S_0, S_\text{f}, p) \coloneqq \{x\in \mathbb{R}^{\TP} : A(\alpha)x \leq b(S_0, S_\text{f}, p)\}
    \end{equation}
    with $A(\alpha) \in \mathbb{R}^{4\TP\times \TP}$ and $b(S_0, S_\text{f}, p) \in \mathbb{R}^{4\TP}$ defined by
        \begin{subequations}
    \begin{align}
		&A(\alpha) \coloneqq
		\left(-I , I , \Gamma^\top , -\Gamma^\top
		\right)^\top\;\text{and}\\
		&b(S_{0}, S_\text{f}, p) \coloneqq   
		\left(-\underline{x}\mathbf{1}_\TP^\top, \overline{x}\mathbf{1}_\TP^\top, \right. \nonumber \\
        &\left. \frac{(\overline{S}\mathbf{1}_\TP - S_0a_{\TP})^\top}{\Delta t}, \frac{(S_0a_{\TP-1} - \underline{S}\mathbf{1}_{d-1})^\top}{\Delta t}, \frac{\alpha^{\TP}S_{0} - S_\text{f}}{\Delta t}\right)^\top.
	\end{align}
\end{subequations}
Moreover, we have $a_\TP \coloneqq (\alpha, \alpha^2, \ldots, \alpha^\TP)^\top$, $I \in \mathbb{R}^{\TP \times \TP}$ is the identity matrix, and $\Gamma \in \mathbb{R}^{\TP \times \TP}$ is a Toeplitz matrix with first column and row defined by $(1, \alpha, \ldots, \alpha^{\TP-1})^\top$ and $(1, 0, \ldots, 0)$, respectively. 
Note that we use $x$ as flexibility variable following \cite{b5,b6,b11,b14,b17,b18} rather than the notation $u$ which is commonly used in systems and control. 
The polytopes $\mathcal{B}(S_0, S_\text{f}, p)$ model a variety of real-world flexibilities such as batteries and thermostatically controlled loads, cf.~\cite{b6,b7}. 
For example, if $\overline{x} > 0$, $\underline{x} < 0$, $\underline{S} < \overline{S}$, and $\alpha^d S_0 \geq \underline{S}$ then $\mathcal{B}(S_0, \underline{S}, p)$ satisfies Assumptions \ref{assumption} and~\ref{assumption2}. 
For alternative energy storage formulations, we refer to \cite{b28}. 

Our approach aims to compute certain vectors of extreme actions within the polytopes.
\begin{definition}[Extreme actions]\label{def:y}
    Let polytopes $\mathcal{P}(A_i, b_i) \subset \mathbb{R}^\TP$, $i \in \{1,\ldots,\NF\}$ satisfy the Assumptions~\ref{assumption} and~\ref{assumption2}.
    Then, for $j \in \{-1, 1\}^{\TP}$ the vectors $y_i^{j} \in \mathbb{R}^{\TP}$ defined by
    \begin{equation}\label{eq:yi0}
        y_{i,1}^{j} := j_1 \cdot \max \{ j_1 \cdot x \in \mathbb{R} : (x, \mathbf{0}_{\TP-1})^\top \in \mathcal{P}(A_i, b_i)\},
    \end{equation}
    and
    \begin{equation}\label{eq:yi}
	    y_{i, t}^{j} \coloneqq j_t \cdot \max \{ j_t \cdot x \in \mathbb{R} : (y^j_{i, [t-1]}, x, \mathbf{0}_{\TP-t})^\top \in \mathcal{P}(A_i, b_i) \}.
    \end{equation}
    for $t \in \{2,\ldots,\TP\}$ are called extreme actions.
\end{definition}
Note that $j_t = -1$ in \eqref{eq:yi0} and \eqref{eq:yi} is equivalent to replacing the maximization with a minimization.
Intuitively, the vectors $y^j_{i}$ are obtained by moving as far as possible in each axis in the negative direction if $j_t = -1$, and in the positive direction if $j_t = 1$, cf. Fig.~\ref{fig:exampel_algo}.
The vectors $y^j_i$ exist for all polytopes fulfilling the Assumptions \ref{assumption} and \ref{assumption2}, and it holds that $y^j_i \in \mathcal{P}(A_i, b_i)$ by construction.
The summation over all $i = 1, \ldots, \NF$ with fixed $j \in \{-1, 1\}^{\TP}$ is denoted by:
\begin{equation}\label{eq:y}
	v^j \coloneqq \sum_{i=1}^{\NF} y_{i}^{j},
\end{equation}
and the convex hull of the set of summed vectors leads to
\begin{equation}\label{eq:approx}
	\mathcal{A} \coloneqq \textnormal{Conv}(\{v^j : j \in \{-1, 1\}^\TP\}).
\end{equation}
The set $\mathcal{A}$ can be described as a deformed cuboid, cf. Fig.~\ref{fig:exampel_algo}.
It follows from \eqref{eq:approx} that $\mathcal{A}$ is a polytope and $\mathcal{A} \subseteq \mathcal{M}$.
\begin{figure}[tb]
    \centering
    \centerline{\includegraphics[width=1.0\columnwidth]{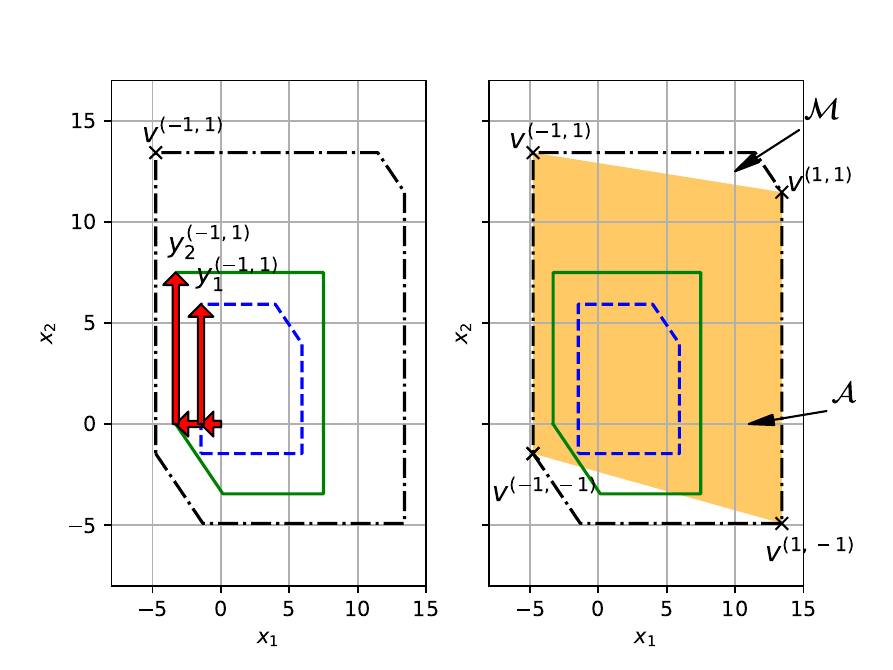}}
    \caption{Left: vectors $y^{(-1, 1)}_1$, $y^{(-1, 1)}_2$ within the polytopes shown in dashed blue and solid green, and the sum $v^{(-1, 1)}$ shown in the Minkowski sum $\mathcal{M}$ in dash-dotted black. Right: all possible vectors $v^j, j \in \{-1, 1\}^2$ in the Minkowski sum with the resulting set $\mathcal{A}$ in orange.}
    \label{fig:exampel_algo}
\end{figure}

Henceforth, we show that the summed extreme actions \eqref{eq:y} are distinct vertices of the Minkowski sum.
Thus, the convex hull $\mathcal{A}$ of the summed vectors is an inner approximation of the Minkowski sum $\mathcal{M}$.

\section{Main Results}\label{sec:3}

Next, we discuss the properties of the summed extreme actions.
Due to space limitations, standard definitions such as convex independence and vertex are not given; instead we refer to, e.g., \cite{b22,b23}.
Fig.~\ref{fig:exampel_algo} illustrates the setting analysed in the following technical results. The proofs of the lemmas and propositions are given in Appendix\ref{sec:appendixA} to Appendix\ref{sec:AppnedixD}.
\begin{lemma}\label{lemma:properties}
Let polytopes $\mathcal{P}(A_i, b_i) \subset \mathbb{R}^\TP$, $i \in \{1,\ldots,\NF\}$, fulfill Assumptions~\ref{assumption} and~\ref{assumption2}. Further, let $v^j, v^k \in \mathbb{R}^\TP, j,k \in \{-1, 1\}^\TP$ satisfy \eqref{eq:y}. Then, the following 
holds:
\begin{enumerate}
    \item $v^j_{t} \begin{cases}
            \geq 0 & \text{for $j_t = 1$} \\
            \leq 0 & \text{for $j_t = -1$}
        \end{cases}   \qquad\forall\,t \in \{1,\ldots,\TP\}$,
    \item if $j \neq k$, then $v^j \neq v^k$.
\end{enumerate}
\end{lemma}
\begin{lemma}\label{lemma_not_in_M_sum}
	Let polytopes $\mathcal{P}(A_i, b_i) \subset \mathbb{R}^\TP$, $i \in \{1, \ldots, \NF\}$ with Minkowski sum $\mathcal{M}$ fulfill the Assumptions \ref{assumption} and \ref{assumption2}, and $p\in \mathbb{R}^\TP$. Further, let $v^j \in \mathbb{R}^\TP, j \in \{-1, 1\}^\TP$ satisfy \eqref{eq:y}. For $t \in \{2, \ldots, \TP\}$, if $p_{[t-1]} = v^j_{[t-1]}$ and $p_{t} > v^j_t$ with $j_t=1$, or $p_{t} < v^j_t$ with $j_t = -1$, then $p \not\in \mathcal{M}$.
    Furthermore, if $p_1 > v^j_1$ with $j_1 = 1$ or $p_1 < v^j_1$ with $j_1 = -1$, then $p \notin \mathcal{M}$.
\end{lemma}
Lemma \ref{lemma_not_in_M_sum} states that there can be no vector in $\mathcal{M}$ that has $t-1$ coordinates equal to $v^j$ and a value greater than $v^j_t$ in the $t$-th coordinate if $j_t = 1$.
Similarly, there cannot be a vector with equal $t-1$ coordinates in $\mathcal{M}$ that has a value less than $v^j_t$ when $j_t = -1$.
This characteristic behavior is also illustrated in Fig.~\ref{fig:exampel_algo}.

\begin{proposition}\label{lemma1}
    Let polytopes $\mathcal{P}(A_i, b_i) \subset \mathbb{R}^\TP$, $i \in \{1, \ldots, \NF\}$, fulfill the Assumptions \ref{assumption} and \ref{assumption2}. Further, let $v^j \in \mathbb{R}^\TP, j\in \{-1, 1\}^\TP$ satisfy \eqref{eq:y}, and $\mathcal{A}$ satisfy \eqref{eq:approx}. Then, $v^j$ is a vertex of~$\mathcal{A}$.
\end{proposition}
The proposition states that $v^j$ is a vertex of $\mathcal{A}$, and by Lemma~\ref{lemma:properties}, the elements of $\{v^j : j \in \{-1, 1\}^\TP\}$ are distinct.
Thus, they are distinct vertices of $\mathcal{A}$.

\begin{proposition}\label{lemma2}
	Let polytopes $\mathcal{P}(A_i, b_i) \subset \mathbb{R}^\TP$, $i \in \{1, \ldots, \NF\}$, with Minkowski sum $\mathcal{M}$ fulfill the Assumptions \ref{assumption} and \ref{assumption2}. Further, let $v^j \in \mathbb{R}^\TP, j \in \{-1, 1\}^\TP$ satisfy \eqref{eq:y}, $\mathcal{A}$ satisfy \eqref{eq:approx}, and $p, q \in \mathcal{M}$ with $v^j= t p+(1- t)q, t \in (0, 1)$, then, $p, q \in \mathcal{A}$.
\end{proposition}
The proposition gives that if $v^j$ is a proper convex combination of elements $p, q \in \mathcal{M}$, then $p, q$ must be in $\mathcal{A}$, see Appendix\ref{sec:AppnedixD} for the proof.
We can now state our first main result which shows that the readily computable $2^d$ vectors $v^j$ are indeed vertices of the Minkowski sum, and, thus, their convex hull constitutes  an inner approximation.

\begin{theorem}[Extreme actions define vertices]\label{theorem:Vertices of M-sum}
	Let polytopes $\mathcal{P}(A_i, b_i) \subset \mathbb{R}^\TP$, $i \in \{1, \ldots, \NF\}$, with Minkowski sum $\mathcal{M}$ fulfill the Assumptions \ref{assumption} and \ref{assumption2}.
    Then, any $v^j \in \mathbb{R}^\TP, j \in \{-1,1\}^\TP$, satisfying \eqref{eq:y} is a vertex of $\mathcal{M}$.
\end{theorem}

\begin{proof}
    Suppose that $v^j$ is not a vertex of $\mathcal{M}$, then $v^j = tp + (1-t)q$ with $p, q \in \mathcal{M}$, $p \neq q$ and $t \in (0, 1)$.
    From Proposition~\ref{lemma2} it follows that $p, q \in \mathcal{A}$, which gives $v^j$ as a proper convex combination of elements in $\mathcal{A}$.
    Thus $v^j$ cannot be a vertex of $\mathcal{A}$, which contradicts Proposition~\ref{lemma1}.
    Hence, the assumption that $v^j$ is not a vertex of $\mathcal{M}$ must be false.
\end{proof}
Theorem \ref{theorem:Vertices of M-sum} combined with Lemma \ref{lemma:properties} states that the sums of extreme actions are distinct vertices of $\mathcal{M}$, providing a novel method to compute a subset of Minkowski sum vertices.
Note that $\mathcal{A}$ is exact for cuboids since they have $2^\TP$ vertices in $\TP$-dimensional space and $\lvert \{v^j : j \in \{-1, 1\}^\TP\}\rvert = 2^\TP$.
Moreover, for any set $\mathcal{V} \subseteq \{v^j : j \in \{-1, 1\}^\TP\}$ it holds that $\textnormal{Conv}(\mathcal{V}) \subseteq \mathcal{M}$, thus $\textnormal{Conv}(\mathcal{V})$ is an inner approximation of $\mathcal{M}$.

\section{Application to Energy Storage}\label{sec:4}
Next, we present an efficient algorithm for computing the extreme actions $y^j, j \in \{-1, 1\}^\TP$ for energy storages $\mathcal{B}(S_0, \underline{S}, p)$.
This approach is then further extended by a corrective algorithm to compute a subset of vertices for the corresponding energy storages $\mathcal{B}(S_0, S_f, p)$. 
Note that this is necessary because $\mathcal{B}(S_0, S_f, p)$ may violate Assumption~\ref{assumption2} for arbitrary $S_f > \underline{S}$.
Finally, the complete algorithm for polytopes $\mathcal{B}(S_{0, i}, S_{f, i}, p_i), i = 1, \ldots, \NF$ is presented. 

Algorithm \ref{algo:batt} computes the $y^j$ for given parameters $S_0, p$, $j \in \{-1, 1\}^\TP$, and $S_f = \underline{S}$ without invoking any numerical optimization problems.
The procedure iterates through the components $j_t$ of $j$. If $j_t = 1$, then $y^j_t$ is determined by charging to the limit without violating the constraints, and by discharging to the limit for $j_t = - 1$.
To this end, Line 5 checks whether the upper energy constraint for $y^j_t = \overline{x}$ is violated.
If so, $y^j_t$ in Line 6 is chosen to fully charge the battery. 
Similarly, Line 10 checks whether the lower energy constraint is violated with $y^j_t = \underline{x}$.
If so, $y^j_t$ in Line 11 is chosen to fully discharge the battery.
The application of Algorithm \ref{algo:batt} to all $j \in \{-1, 1\}^\TP$ yields the set of vectors $\{y^j : j \in \{-1, 1\}^\TP\}$.
The aggregated vectors are then obtained by storing these vectors in matrices $V_i$ and further calculating $\sum_{i=1}^\NF V_i$. \vspace*{1mm}
\begin{algorithm}[ht]
	\caption{(Vertex)}\label{algo:batt}
    \textbf{Input} $S_{0}, p, j \in \{-1, 1\}^\TP$
	\begin{algorithmic}[1]
        \State $y^j \gets \mathbf{0}_\TP$
        \For{$t = 1$ \textbf{to} $\TP$}
            \If{$j_t = 1$}
                \State  $y^j_t \gets \overline{x}$
                \If{$\alpha^t S_0 + \sum_{\tau = 1}^{t}\alpha^{t-\tau}y^j_\tau\Delta t > \overline{S}$}
                    \State $y^j_t \gets \frac{\overline{S} - (\alpha^t S_0 + \sum_{\tau = 1}^{t-1}\alpha^{t-\tau}y^j_\tau \Delta t)}{\Delta t}$
                \EndIf
            \ElsIf{$j_t = -1$}
                \State $y^j_t \gets \underline{x}$
                \If{$\alpha^t S_0 + \sum_{\tau = 1}^{t}\alpha^{t-\tau}y^j_\tau\Delta t < \underline{S}$}
                    \State $y^j_t \gets \frac{\underline{S}-(\alpha^t S_0 + \sum_{\tau = 1}^{t-1}\alpha^{t-\tau}y^j_\tau \Delta t)}{\Delta t}$
                \EndIf
            \EndIf
        \EndFor
    \end{algorithmic}
    \textbf{Output} $y^j$
\end{algorithm}

The previous approach can be extended to compute a subset of vertices for energy storages $\mathcal{B}(S_{0}, S_{f}, p)$, which may violate Assumption~\ref{assumption2} for arbitrary $S_{f} > \underline{S}$, cf. Fig.~\ref{fig:Exmaple_Battery}.
\begin{figure}[tb]
	    \centering
	    \includegraphics[width=\columnwidth]{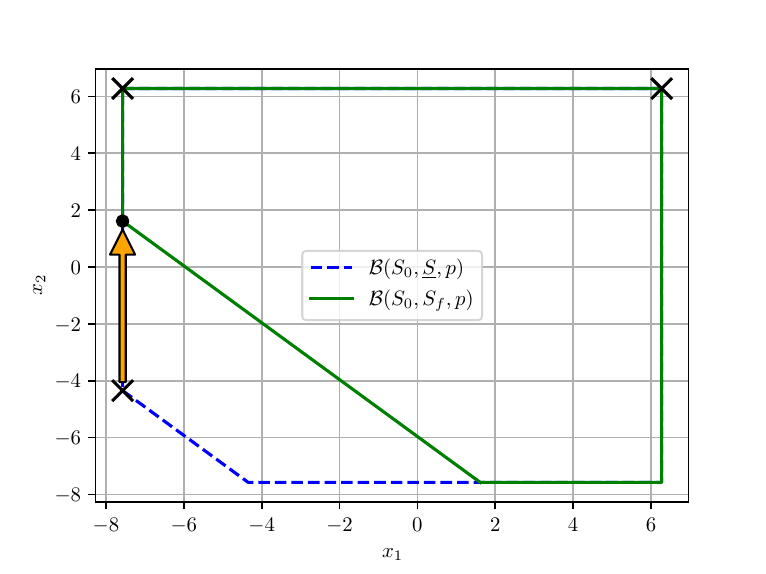}
	    \caption{The set $\mathcal{B}(S_0, S_f, p)$ in solid green and the set $\mathcal{B}(S_0, \underline{S}, p)$ in dashed blue. The crosses on $\mathcal{B}(S_0, \underline{S}, p)$ indicate the $y^j$, and the arrow with dot visualizes the correction process.} \label{fig:Exmaple_Battery}
	\end{figure}
Suppose $y^{j}$ satisfying Definition~\ref{def:y} is obtained for the energy storage $\mathcal{B}(S_{0}, \underline{S}, p)$.
The set $\{ x \in \mathbb{R} : (y^j_{[\TP-1]}, x)^\top \in \mathcal{B}(S_{0}, S_{f}, p) \}$ is equivalent to $\{x \in \mathbb{R} : \underline{x} \leq x \leq \overline{x},\ S_{f} \leq \alpha^\TP S_{0}+\sum_{\tau = 1}^{\TP-1}\alpha^{\TP-\tau} y_{\tau}^{j}\Delta t + x\Delta t \leq \overline{S}\}$. Hence, if $y^j \notin \mathcal{B}(S_{0}, S_{f}, p)$, then $S_{f} > \alpha^\TP S_{0} + \sum_{\tau = 1}^{\TP} \alpha^{\TP-\tau}y_{\tau}^{j}\Delta t$ since $y^j \in \mathcal{B}(S_{0}, \underline{S}, p)$ and the remaining inequalities are identical for both sets.
Thus, increasing the values in $y^j$ without violating the power constraints so that the inequality associated with $S_f$ is satisfied yields $y^j \in \mathcal{B}(S_{0}, S_f, p)$.
The polytope $\mathcal{B}(S_{0}, S_f, p)$ models an energy storage with a minimum final energy constraint.
The charging and discharging in $\mathcal{B}(S_{0}, \underline{S}, p)$ may result in a final energy less than $S_f$.
Thus, by correcting---i.e., increasing the values in $y^j$---one can achieve the given final energy $S_f$, and the corrected vectors $\Tilde{y}^j$ fulfill $\Tilde{y}^j \in \mathcal{B}(S_0,S_f,p)$, cf. Lemma~\ref{lemma:convergence}.
The correction process starts at the last period $\TP$ by checking in Line 2 of Algorithm~\ref{algo:positive correction batt} whether $S_f$ can be reached without violating the power constraints.
If this is possible, the $\TP$-th coordinate of $y^j$ is changed in Line 3 and the algorithm terminates, otherwise, the coordinate $\TP-1$ is changed to the highest possible value $\overline{x}$ and it is checked again whether the final SoC can be reached without violating the power constraints.
If possible, the $\TP$-th coordinate of $y^j$ is changed in Line 9 and the algorithm terminates, otherwise, it is continued with the coordinate $\TP-2$ and so forth until $S_f$ is reached.
This correction process is visualized in Fig.~\ref{fig:Exmaple_Battery}. First, the $y^j$ are computed within $\mathcal{B}(S_{0}, \underline{S}, p)$, i.e., the four crosses in Fig. \ref{fig:Exmaple_Battery}, then the coordinates of the crosses not contained in $\mathcal{B}(S_{0}, S_f, p)$ are increased, so that $S_f$ is reached, indicated by the arrow and dot in Fig. \ref{fig:Exmaple_Battery}.
\begin{algorithm}[tb]
	\caption{(Correction)}\label{algo:positive correction batt}
    \textbf{Input} $S_{0}, S_{f}, p, y^j \in \mathcal{B}(S_{0}, \underline{S}, p)$
	\begin{algorithmic}[1]
        \If{$S_{f} > \alpha^\TP S_{0} + \sum_{\tau = 1}^{\TP} \alpha^{\TP-\tau}y_{\tau}^{j}\Delta t$}
        \If{$\underline{x} \leq \frac{S_{f} - (\alpha^\TP S_{0} + \sum_{\tau = 1}^{\TP-1}\alpha^{\TP-\tau} y^j_{ \tau}\Delta t)}{\Delta t} \leq \overline{x}$}
        \State $y_{\TP}^{j} \gets
        \frac{S_{f} - (\alpha^\TP S_{0} + \sum_{\tau = 1}^{\TP-1} \alpha^{\TP-\tau}y^j_{ \tau}\Delta t)}{\Delta t}$
        \EndIf
        \State $t \gets \TP - 1$
        \While{$S_{f} \neq \alpha^\TP S_{0} + \sum_{\tau = 1}^{\TP}\alpha^{\TP-\tau} y^j_{ \tau}\Delta t$ and $t > 0$}
        \State $y_{t}^{j} \gets \overline{x}$
        \If{$\underline{x} \leq \frac{S_{f} - (\alpha^\TP S_{0} + \sum_{\tau = 1}^{\TP-1}\alpha^{\TP-\tau} y^j_{ \tau}\Delta t)}{\Delta t} \leq \overline{x}$}
        \State $y_{\TP}^{j} \gets \frac{S_{f} - (\alpha^\TP S_{0} + \sum_{\tau = 1}^{\TP-1} \alpha^{\TP-\tau}y^j_{ \tau}\Delta t)}{\Delta t}$
        \EndIf
        \State $t \gets t - 1$
        \EndWhile
        \EndIf
        \State $\Tilde{y}^j \gets y^j$
    \end{algorithmic}
    \textbf{Output} $\Tilde{y}^j$
\end{algorithm}
\begin{lemma}\label{lemma:convergence}
    Let $\mathcal{B}(S_{0}, \underline{S}, p)$ with parameter vector $p = (\alpha, \underline{x}, \overline{x}, \underline{S}, \overline{S}, \Delta t)^\top$ fulfill Assumptions \ref{assumption} and \ref{assumption2}. Further, let $y^j \in \mathbb{R}^\TP$, $j \in \{-1, 1\}^\TP$ satisfy Definition~\ref{def:y} for $\mathcal{B}(S_{0}, \underline{S}, p)$. Then, $\Tilde{y}^j$ defined by Algorithm~\ref{algo:positive correction batt} is in $\mathcal{B}(S_{0}, S_{f}, p)$ if $\mathcal{B}(S_{0}, S_{f}, p) \neq \emptyset$.
\end{lemma}
For the proof see Appendix\ref{sec:AppendixE}. The next result shows that the corrected vector $\Tilde{y}^j$ is a vertex of $\mathcal{B}(S_{0}, S_{f}, p)$. The proof is given in Appendix\ref{sec:AppendixF}.
\begin{theorem}\label{VertexNonAssumptionBatteries}
    Let $\mathcal{B}(S_{0}, \underline{S}, p) \subset \mathbb{R}^\TP$ with parameter vector $p = (\alpha, \underline{x}, \overline{x}, \underline{S}, \overline{S}, \Delta t)^\top$ fulfill Assumptions \ref{assumption} and \ref{assumption2}, and $\mathcal{B}(S_{0}, S_f, p)$ be nonempty. Further, let $y^j \in \mathbb{R}^\TP$, $j \in \{-1, 1\}^\TP$ satisfy Definition~\ref{def:y} for $\mathcal{B}(S_{0}, \underline{S}, p)$. Then, $\Tilde{y}^j$ defined by Algorithm~\ref{algo:positive correction batt} is a vertex of $\mathcal{B}(S_{0}, S_{f}, p)$.
\end{theorem}

Algorithm \ref{algo:batt_complete} uses the polytopes $\mathcal{B}(S_{0, i}, \underline{S}, p_i), i = 1, \ldots, \NF$ to compute a subset of their vertices, and then corrects these $y_i^j$ with Algorithm~\ref{algo:positive correction batt} such that vertices of polytopes $\mathcal{B}(S_{0, i}, S_{f, i}, p_i), i = 1, \ldots, \NF$ are obtained.
The parameter $g$ in Algorithm \ref{algo:batt_complete} allows considering a subset of $\{v^j: j \in \{-1, 1\}^\TP\}$. This allows adjusting accuracy and computational complexity. 
Line 1 guarantees that $g$ is limited to $2^\TP$ and ensures that all $2^d$ vectors for up to $d=8$ dimensions are included. 
If $g < 2^\TP$ and $d > 8$, we propose to stochastically select the $j \in \{1,-1\}^d$ using a uniform distribution.
In Line~7, the vector of zeros modeling the non-use of flexibility is inserted into the $g+1$-th column of $V$.
This is not necessary when $g = 2^\TP$ because then all vectors are computed, and it can be shown that in this case $\mathbf{0}_\TP \in \mathcal{A}$.
Note that the vector is to be appended only if $\alpha_i^\TP S_{0, i} \geq S_{f, i}\ \forall i \in \{1, \ldots, \NF\}$.
Otherwise, $\mathbf{0}_\TP$ is not included in $\mathcal{B}(S_{0, i}, S_{f, i}, p_i)$.
In Lines 15 and 16, the Algorithms \ref{algo:batt} and \ref{algo:positive correction batt} are invoked, respectively. 
Finally, it should be noted that the for loop in Line 11 only needs to be executed once for energy storage devices with identical parameters, as the inner for loop (Line 14) would result in the same matrix $V_i, \forall i \in \{1, \ldots, n\}$. 
In this case, $V_i$ can be calculated once and multiplied by the number of devices $n$.
On the supplementary website \cite{b38} we provide the complete Python code together with examples and illustrations.

\begin{algorithm}[tb]
	\caption{(Complete Algorithm)}\label{algo:batt_complete}
    \textbf{Input} $S_{0, i}, S_{f, i}, p_{i}, i=1, \ldots, \NF$, $g$
	\begin{algorithmic}[1]
        \If{$g < 2^\TP$ and $d > 8$}
            \State init $\mathcal{J}$ \Comment{generate $g$ distinct elements in $\{-1, 1\}^\TP$}
        \Else
            \State $\mathcal{J} \gets \{-1, 1\}^\TP$
        \EndIf
        \If{$\alpha_i^\TP S_{0, i} \geq S_{f, i} \ \forall i \in \{1, \ldots, \NF\}$}
        	\State $V \gets \mathbf{0}_{\TP \times (g+1)}$ \Comment{$\TP \times (g+1)$ matrix of zeros}
        \Else
        	\State $V \gets \mathbf{0}_{\TP \times g}$ \Comment{$\TP \times g$ matrix of zeros}
        \EndIf
        \For{$i = 1$ \textbf{to} $\NF$}
            \State $V_i \gets \mathbf{0}_{\TP \times g}$
            \State $k \gets 1$
            \For{$j \in \mathcal{J}$}
                \State $y^j_i \gets \text{Vertex}(S_{0,i}, p_i, j)$
                \State $\Tilde{y}^j_i \gets \text{Correction}(S_{0,i}, S_{f, i}, p_i, y^j_i)$
                \State $V_i[:, k] \gets \Tilde{y}^j_i$
                \State $k \gets k + 1$
            \EndFor
        \EndFor
        \State $V[:, 1:g] \gets \sum_{i=1}^\NF V_i$
    \end{algorithmic}
    \textbf{Output} $V$
\end{algorithm}

\section{Benchmark Results}\label{sec:5}
Now we compare the proposed method using the benchmark for Minkowski sum approximations previously published in \cite{b17}. 
The considered scenario models households with real demand curves and stationary batteries modeled by the polytopes $\mathcal{B}(S_{0, i}, \frac{1}{2}S_{0, i}, p_i), p_i = (1, \overline{x}_i, \underline{x}_i, 0, \overline{S}_i, \frac{1}{4})$.
The battery parameters are sampled from intervals: $\overline{S}_i \in [10.5, 13.5]$ (KWh), $S_{0,i} \in [0, 10.5]$ (kWh), $\overline{x}_i \in [4, 6]$ (kW), and $\underline{x}_i \in [-6, -4]$ (kW) $\forall i \in \{1, \ldots, \NF\}$, cf. \cite{b17} for details and for an indepth discussion of the benchmark. 
With these parameters, typically at least 60\% of the polytopes generated violate Assumption~\ref{assumption2}.

We assess the quality of inner approximations via the Unused Potential Ratio (UPR) defined as
\begin{equation}\label{eq:UPR}
    \text{UPR} \coloneqq \frac{z_\text{approx} - z_\text{exact}}{z_\text{no flex} - z_\text{exact}}\cdot 100.
\end{equation}
Here $z_\text{approx}$ represents the solution of an optimization problem, e.g., the minimal cost or peak power, based on the approximation, $z_\text{exact}$ the solution to the same problem without aggregation using all constraints at once, i.e., the exact feasible region, and $z_\text{no flex}$ the solution in a setting without flexibility. 
If the UPR is close to $0\ \%$, then the approximation and the Minkowski sum yield almost the same result.
Otherwise, if the UPR is close to $100\ \%$, then there is a large (unused) improvement available in the approximation. 
There is also the possibility that the UPR value is greater than $100\ \%$, in which case the solution without flexibility, i.e., $\mathbf{0}_\TP$, gives better results than using the approximation.
We consider the objectives $c^{\top}\left(x + \sum_{i=1}^{N}q_{\text{i}}\right)\Delta t$
for economic cost and $\left\lVert x + \sum_{i=1}^{N}q_{\text{i}}\right\rVert_{\infty}$ for peak power, where $c$ is the associated cost and $q_i$ the household demand. 
To account for uncertainties, UPR values are calculated for each month of a year along with $5$ random villages, i.e., sets of households with stationary batteries, the median of which is used for further analysis.

We use Algorithm \ref{algo:batt_complete} and a uniformly sampled subset $\mathcal{J} \subseteq \{-1, 1\}^\TP$ with $\TP^2$ distinct vectors to calculate the $y^j$, i.e., $g = \TP^2$.
Note that $g$ needs to be a function of $\TP$, as the number of vertices increases with increasing dimension, e.g., a hypercube has $2^\TP$ vertices in $\TP$ dimensions.
To motivate the quadratic dependence, we conducted an experiment with $100$ batteries for $12, 24, \ldots 96$ time periods.
For each tuple $(\NF,\TP) \in \{100\} \times \{12, 24, \ldots, 96\}$, the approximation and UPR values are calculated 50 times to measure the variation in UPR values with different choices of $\mathcal{J} \subseteq \{1,-1\}^d$. The results are shown in Fig. \ref{fig:Results_Distribution_boxplot}. 
The maximum range of the UPR values for the peak power and cost objectives is $6.3 \ \%$ and $14.9 \ \%$, respectively, indicating a good degree of robustness with different choices of $\mathcal{J} \subseteq \{1,-1\}^d$. 
Spontaneous fluctuations in the UPR values of Fig. \ref{fig:Results_Distribution_boxplot} are most likely induced by the random nature of the battery parameters, as the parameters determine the shape of the polytopes and thus the quality of the approximation. 
Since Algorithm \ref{algo:batt_complete} computes $\TP^2$ vertices for  each of the $\NF$ devices, the total number of vectors to compute is $\NF\TP^2$. 
Thus, the computational complexity is quadratic in the number of time periods and linear in the number of devices.
\begin{figure}[tb]
    \centering
    \includegraphics[width=1.1\columnwidth]{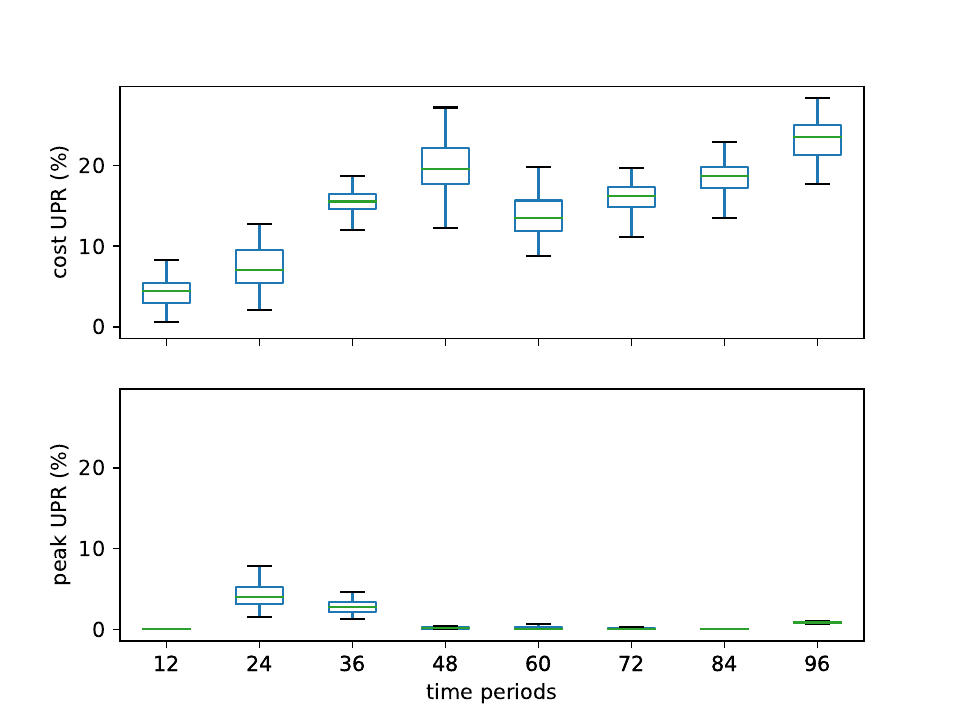}
    \caption{Boxplot for UPR values with 100 batteries, $d=12, 14, \ldots, 96$ time periods, and $g=d^2$. For each time period, the approximation is calculated 50 times.}
    \label{fig:Results_Distribution_boxplot}
\end{figure}
\begin{figure*}[tb]
    \centering
    \includegraphics[width=0.9\textwidth,height=0.45\textheight]{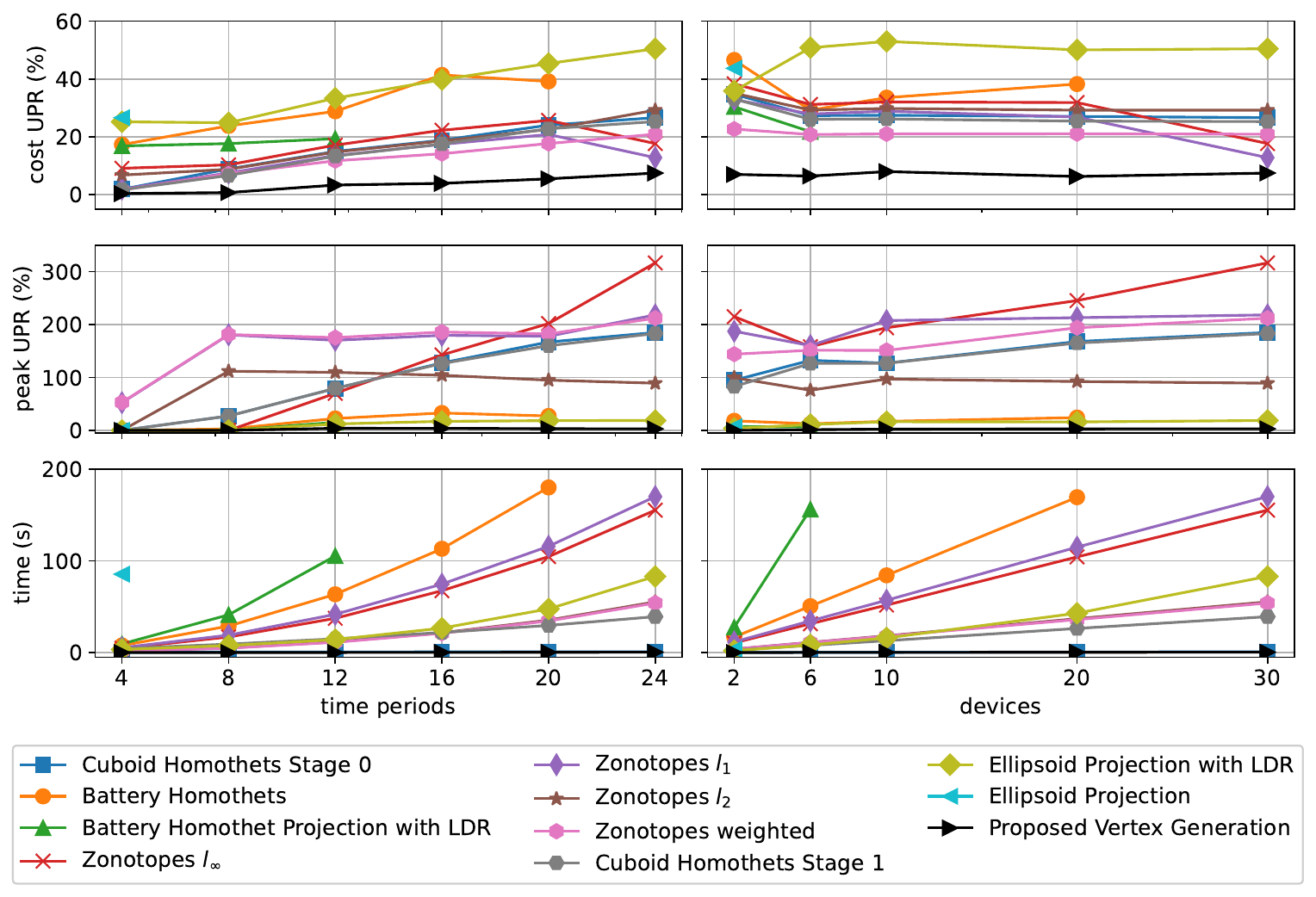}
    \caption{Results for experiments with tuples $(n,d) \in \{30\} \times \{4, 8, 12, 16, 24\}$ in the first column and experiments with tuples $(n,d) \in \{2, 6, 10, 20, 30\} \times \{24\}$ in the second column. The cost UPR values are shown in the first row, the peak UPR values in the second row and the calculation times in the third row.}
    \label{fig:res_algos}
\end{figure*}
\begin{table}[tb]
    \caption{Max UPR values and max calculation time for 4, 8, 12, 16, 20, 24 time periods and 2, 6, 10, 20, 30 batteries for different inner approximation methods}
    \setlength{\tabcolsep}{3pt}
    \centering
    \begin{tabular}{|c|c|c|c|c|c|}
    \hline
        Algorithm & Ref. & Time (s) & \multicolumn{2}{|c|}{UPR (\%)}\\
    \hline
         & & & Peak & Cost\\
    \hline
        Cuboid Homothets Stage 0 & \cite{b13} & 0.80 & 184.94 & 34.71\\
         Battery Homothets & \cite{b7} & - & - & - \\
        \begin{tabular}{@{}c@{}}Battery Homothet Projection \\ with LDR\end{tabular} & \cite{b15} & - & - & -\\
         Zonotopes $l_\infty$ & \cite{b11} & 155.29 & 316.41 & 38.31\\
         Zonotopes $l_1$ & \cite{b11} & 170.08 & 354.29 & 32.88\\
         Zonotopes $l_2$ & \cite{b11} & 55.23 & 168.10 & 35.04\\
         Zonotopes weighted & \cite{b12} & 53.97 & 228.23 & 22.66\\
         Cuboid Homothets Stage 1 & \cite{b13} & 39.11 & 183.01 & 33.04\\
         Ellipsoid Projection with LDR & \cite{b14} & 83.05 & 18.77 & 52.97\\
         Ellipsoid Projection & \cite{b5} & - & - & -\\
        \textbf{Proposed Vertex Generation} &  & \textbf{0.38} & \textbf{4.92} & \textbf{7.95}\\
    \hline
    \end{tabular}
    \label{tab1}
\end{table}
\begin{table}[tb]
    \caption{Maximum UPR values and max calculation time for 12, 24, $\ldots$, 96 time periods and 50, 100, $\ldots$, 500 batteries.}
    \setlength{\tabcolsep}{3pt}
    \centering
    \begin{tabular}{|c|c|c|c|c|}
    \hline
        Algorithm & Time (s) & \multicolumn{2}{|c|}{UPR (\%)}\\
    \hline
         &  & Peak & Cost\\
    \hline
        \textbf{Proposed Vertex Generation} & \textbf{361.76} & \textbf{7.37} & \textbf{33.93}\\
    \hline
    \end{tabular}
    \label{tab2}
\end{table}

To evaluate the 11 approximation algorithm, we conducted experiments with 2, 6, 10, 20 and 30 batteries and 4, 8, 12, 16, 20 and 24 time periods.
For each tuple $(\NF,\TP) \in \{2, 6, 10, 20, 30\} \times \{4, 8, 12, 16, 20, 24\}$, the UPR values and calculation times are computed similar to \cite{b17}.
The maximum computation time and the maximum UPR value for cost and peak objectives across all tuples are listed in the columns of Tab. \ref{tab1} for each algorithm. 
The reason for empty entries is a limitation in the benchmark, which skips algorithms that take longer than $10$ minutes for a tuple $(\NF, \TP)$ and are thus not calculated for further tuple combinations. 
In addition, Fig. \ref{fig:res_algos} shows the results for the tuples $(n,d) \in \{30\} \times \{4, 8, 12, 16, 20, 24\}$, i.e., fixed $30$ devices and varying time periods in the first column, and for the tuples $(n,d) \in \{2, 6, 10, 20, 30\} \times \{24\}$, i.e., fixed $24$ periods and varying devices in the second column. 
The cost UPR values are shown in Fig. \ref{fig:res_algos} in the first row, the peak UPR values in the second row, and the calculation times in the third row. 
Tab.~\ref{tab1} with Fig. \ref{fig:res_algos} shows that our proposed vertex generation achieves the lowest calculation times and UPR values. 
This solves one of the problems in \cite{b17}, namely the objective dependent performance, e.g., the algorithm "Ellipsoid Projection with LDR" achieves the second best results at Peak, but the worst results at Cost. 
Another problem identified in \cite{b17} is that the inner approximations may have worse performance than the setting without flexibility, in which case $\text{UPR} > 100 \ \%$.
This behavior is observable in the Peak column of Tab.~\ref{tab1} for all algorithms except for our approach as well as the algorithm "Ellipsoid Projection with LDR".
This problem is tackled in Line 7 of Algorithm~\ref{algo:batt_complete}, where the vector $\mathbf{0}_\TP$ is implicitly added in the $g+1$-th column of $V$.
The last problem mentioned in \cite{b17} is the computational complexity. Indeed, most algorithms listed in Tab.~\ref{tab1} have unrealistic runtimes already for intra-day time periods, cf. \cite{b17}. However, the results of Tab.~\ref{tab2} show that our approach enables efficient aggregations for full-day time periods and 500 batteries with a maximum computation time of about $6$ minutes. 
Furthermore, Tab.~\ref{tab2} shows that our method exhibits a maximum UPR value for the peak power objective of $7.37\ \%$, which occurs at $(400, 12)$ and is lower than that of the other algorithms, considering only $30$ batteries and $24$ time periods, cf. Tab.~\ref{tab1}.
The maximum cost UPR value is $33.93 \%$, which occurs at $(350, 96)$ and is again better than four of the other algorithms in settings with up to $30$ devices and $24$ time periods, cf. Tab.~\ref{tab1}. This could be further improved by considering more than $\TP^2$ vectors.

It is worth investigating the comparison between the proposed method and the centralized approach, i.e., without aggregation.
The solution to the latter problem is denoted as $z_\text{exact}$ in Eq. \eqref{eq:UPR}. 
In an experiment with 500 devices and 96 time periods, the centralized approach achieved computation times of 54.98~s for the cost objective and 96.58~s for the peak objective. 
Gurobi~\cite{b39} was used to solve the centralized optimization problems. 
While this approach performs better than the proposed method (cf. Table \ref{tab2}), it is worth noting that the centralized approach is impractical due to privacy issues, growing computational complexity, and excessive communication overhead. 
Finally, the proposed approach can be implemented on a simple microcontroller, i.e., no commercial solver is needed, and it can be further improved by computing each device in parallel, making the complexity independent of the number of devices. 
For the sake of completeness, it should be mentioned that distributed optimization methods also exist, e.g., in \cite{b37,b40}, which offer an alternative to aggregation and the centralized approach.

\section{Disaggregation for Vertex-based Approximations}\label{sec:6}
Next, we describe a novel disaggregation method for vertex-based approximations.
Once an estimate of the collective flexibility is available, a grid operator, for example, can select a power profile that needs to be distributed (disaggregated) by the aggregator to the individual flexible devices.
Mathematically, disaggregation is the inverse operation of aggregation. Aggregation can be described as a mapping $f: \mathcal{P}(A_1, b_1) \times \mathcal{P}(A_2, b_2) \times \cdots \times \mathcal{P}(A_\NF, b_\NF) \mapsto \mathbb{R}^\TP, f(x_1, \ldots, x_\NF) \coloneqq \sum_{i=1}^\NF x_i$ for polytopes $\mathcal{P}(A_i, b_i) \subset \mathbb{R}^\TP, i = 1, \ldots, \NF$.
However, this mapping is in general not injective and therefore not invertible, since there may be different sets of vectors with equal sum.
Usually, optimization problems are formulated and solved in the literature to obtain feasible vectors whose sum is the aggregated vector, cf. \cite{b12,b5}.
However, in high-dimensional spaces, this is time-consuming and, indeed, for vertex-based approximations it is not necessary.
Each aggregate vector $x$ can be described as a convex combination of its vertices, i.e.,
\begin{equation}\label{eq:disagg}
    x = \sum_{j \in \mathcal{J}}\alpha_jv^j, \sum_{j \in \mathcal{J}}\alpha_j = 1,  \alpha_j \geq 0 \ \forall j \in \mathcal{J}.
\end{equation}
Each $v^j$ in the aggregation is a summation of vertices $y^j_i \in \mathcal{P}(A_i, b_i)$, hence $v^j = \sum_{i=1}^\NF y^j_i$ with $y^j_i \in \mathcal{P}(A_i, b_i)$. Inserting this equality in \eqref{eq:disagg} yields:
\begin{align}
    & x = \sum_{j \in \mathcal{J}}\alpha_j\sum_{i=1}^\NF y^j_i = \sum_{i=1}^\NF\sum_{j \in \mathcal{J}}\alpha_jy^j_i
\end{align}
Thus, the contribution of flexibility $i$ is the sum $\sum_{j \in \mathcal{J}}\alpha_j y^j_i$. Note that the $\alpha_j$ are fixed by the chosen vector $x$, e.g., from a previously performed optimization by the grid operator, and the $y_i^j$ are known from Algorithm \ref{algo:batt}. 
Therefore, the disaggregation reduces to the calculation of the inner sum for each flexibility, cf. Algorithm~\ref{algo:Disaggregation}. 
Note that this calculation only needs to be carried out once if energy storages with identical parameters are considered. 
\begin{algorithm}[tb]
	\caption{(Disaggregation)}\label{algo:Disaggregation}
    \textbf{Input} $\alpha_j, y^j_i, i = 1, \ldots, \NF, \mathcal{J}$
	\begin{algorithmic}[1]
        \State $D \gets \mathbf{0}_{\TP \times \NF}$
        \For{$i = 1$ \textbf{to} $\NF$}
            \State $D[:,i] \gets \sum_{j \in \mathcal{J}}\alpha_jy^j_i$
        \EndFor
    \end{algorithmic}
    \textbf{Output} $D$
\end{algorithm}

\section{Conclusions}\label{sec:7}
This paper proposed a novel vertex-based inner approximation for the collective flexibility of multiple flexible devices.
Our method is applicable to polytopes satisfying two rather mild assumptions. 
For energy storages which violate Assumption 2 we provide an efficient adaption.
In a benchmark, the proposed approach outperforms ten state-of-the-art inner approximations from the literature in terms of computational complexity and in terms of accuracy for different objectives.
In addition, an efficient disaggregation method is proposed, which is applicable to any vertex-based approximation. 
In combination, the presented methods are to the best of the authors' knowledge the first
to provide a computationally efficient mean to (dis-)aggregate typical energy storages in quarter-hourly periods over an entire day with reasonable accuracy for aggregate cost \textit{and} peak power optimization objectives. 

The proposed method is applicable to a class of practically relevant polytopes. 
In future work, we want to extend our approach to non-polytopic sets like energy storages with \mbox{(dis-)}charging efficiencies and restrictions to simultaneous \mbox{(dis-)}charging.
Also, for other practically relevant polytopes like energy storages with limited availability or time-dependent energy constraints, adaptations to the algorithms need to be developed.
In addition, the proposed method should be extended to consider active and reactive power and to facilitate robust aggregation and disaggregation in the presence of uncertainty.
Finally, to increase the optimization performance at the aggregated level, the optimal choice for the vertices subset $\mathcal{J} \subseteq \{1,-1\}^d$ has to be investigated.

\appendices
\subsection{Proof of Lemma 1}\label{sec:appendixA}
\begin{proof}(Property 1)
	For $t = 2, \ldots, \TP$ we have $v^j_t = \sum_{i=1}^{\NF} j_t \cdot \max \{j_t \cdot x \in \mathbb{R} : (y^j_{i, [t-1]}, x, \mathbf{0}_{\TP-t})^\top \in \mathcal{P}(A_i, b_i) \}$. By construction $y^j_i \in \mathcal{P}(A_i, b_i)$, and by Assumption~\ref{assumption2} $\textnormal{Proj}^{t-1}(y_i^j) \in \mathcal{P}(A_i, b_i)$, hence $(y^j_{i, [t-1]}, 0, \mathbf{0}_{\TP-t})^\top \in \mathcal{P}(A_i, b_i)$. Therefore, if $j_t = 1$, then $\max \{x \in \mathbb{R} : (y^j_{i, [t-1]}, x, \mathbf{0}_{\TP-t})^\top \in \mathcal{P}(A_i, b_i) \} \geq 0 \; \forall \ i$, hence $v^j_t \geq 0$. Otherwise, if $j_t = -1$, then $-\max \{-x \in \mathbb{R} : (y^j_{i, [t-1]}, x, \mathbf{0}_{\TP-t})^\top \in \mathcal{P}(A_i, b_i) \} \leq 0 \; \forall \ i$, thus $v^j_t\leq 0$.

    For $t = 1$ we have $v^j_1 = \sum_{i=1}^{\NF}
    j_1 \cdot \max \{j_1 \cdot x \in \mathbb{R} : (x, \mathbf{0}_{\TP-1})^\top \in \mathcal{P}(A_i, b_i) \}$. Assumption~\ref{assumption2} gives $\mathbf{0}_\TP \in \mathcal{P}(A_i, b_i)$, and by the same reasoning it follows that $v^j_1 \geq 0$ if $j_1 = 1$ and $v^j_1 \leq 0$ if $j_1 = -1$.
\end{proof}

\begin{proof}(Property 2)
	Suppose that $v^k_t = v^j_t$ for $k_t \neq j_t$. Without loss of generality let $j_t = 1$ and $k_t=-1$. Then, $v^j_t =  \sum_{i=1}^{\NF} y^j_t = \sum_{i=1}^{\NF} y^k_t = v^k_t$.
    Since $y^j_t \geq 0$ and $y^k_t \leq 0$ (Property 1), we have $y^j_t = y^k_t = 0 \; \forall \ i$.
    This is however impossible as by Assumption~\ref{assumption2} $\textnormal{Proj}^{t-1}(y_i^j) \in \mathcal{P}(A_i, b_i)$, which implies that $(y^j_{i, [t-1]}, 0, \mathbf{0}_{\TP-t})^\top \in \mathcal{P}(A_i, b_i)$. Furthermore, Assumption~\ref{assumption} ensures the existence of an $\varepsilon \in \mathbb{R}\setminus \{0\}$ with $(y^j_{i, [t-1]}, \varepsilon, \mathbf{0}_{\TP-t})^\top \in \mathcal{P}(A_i, b_i)$, which gives at least two elements, and therefore $\max \{x \in \mathbb{R} : (y^j_{i, [t-1]}, x, \mathbf{0}_{\TP-t})^\top \in \mathcal{P}(A_i, b_i) \} \neq -\max \{-x \in \mathbb{R} : (y^k_{i, [t-1]}, x, \mathbf{0}_{\TP-t})^\top \in \mathcal{P}(A_i, b_i) \}$.
    Thus, the assumption that $v^k_t = v^j_t$ must be false, which yields $v^k_t \neq v^j_t$.

    For $j \neq k$, there exists an index $t$ with $j_t \neq k_t$ and by the above reasoning holds $v_t^k \neq v_t^j$, hence $v^j \neq v^k$.
\end{proof}

\subsection{Proof of Lemma 2}\label{sec:AppendixB}
\begin{proof}
	Assume that $p\in \mathcal{M}$ and $t > 1$, then there are $p_i \in \mathcal{P}(A_i, b_i)$ with $p=\sum_{i=1}^{\NF}p_i$, and $p_{[t-1]} = \sum_{i=1}^{\NF}p_{i, {[t-1]}}=\sum_{i=1}^{\NF}y^j_{i, [t-1]} = v^j_{[t-1]}$. We distinguish two cases:

	\noindent \textit{Case 1}: Let $p_{i, [t-1]} = y^j_{i, [t-1]} \ \forall \ i$. If $p_{t} > v^j_t$ and $j_t = 1$, then there is a $k \in \{1, \ldots, \NF\}$ with $p_{k, t} > y^j_{k, t}$. Since $p_{k, [t-1]} = y^j_{k, [t-1]}$ and $y^j_{k, t} = \max \{ x \in \mathbb{R} : (y^j_{k, [t-1]}, x, \mathbf{0}_{\TP-t})^\top \in \mathcal{P}(A_k, b_k) \}$ it follows that $(p_{k,[t-1]}, p_{k, t},\mathbf{0}_{\TP-t})^\top \not \in \mathcal{P}(A_k, b_k)$, thus $\text{Proj}^t(p_{k}) \not \in \mathcal{P}(A_i, b_i)$. Assumption \ref{assumption2} yields $p_k \not\in \mathcal{P}(A_k, b_k)$, which contradicts the assumption $p \in \mathcal{M}$.
    If $p_t < v^j_t$ and $j_t = -1$, then by similar reasoning it follows that $p\not \in \mathcal{M}$.

	\noindent \textit{Case 2}: $\exists$ $l, k \in \{1, \ldots, \NF\}$ with $p_{l, [t-1]} \neq y^j_{l, [t-1]}$ and $p_{k, [t-1]} \neq y^j_{k, [t-1]}$. Note that the negation of Case 1 yields at least two indices $l, k \in \{1, \ldots, \NF\}$, and a minimum index $m \in \{1, \ldots, \TP\}$ with $p_{l, m} < y^j_{l, m}$ and $p_{k, m} > y^j_{k, m}$.
    For $m \neq 1$ we have $p_{l,[m-1]} = y^j_{l, [m-1]}$ and $y^j_{l, [m-1]} = p_{k, [m-1]}$. If $j_m = -1$ then $\text{Proj}^m(p_l) \not \in \mathcal{P}(A_l, b_l)$ and by Assumption~\ref{assumption2} $p_l \not \in \mathcal{P}(A_l, b_l)$.
    Otherwise, if $j_m = 1$, then $\text{Proj}^m(p_k) \not\in \mathcal{P}(A_k, b_k)$ and by Assumption \ref{assumption2} $p_k \not\in \mathcal{P}(A_k, b_k)$. For $m = 1$ we have that $p_{l, 1} < y^j_{l, 1}$ and $p_{k, 1} > y^j_{k, 1}$. If $j_1 = 1$, then $p_k \notin \mathcal{P}(A_k, b_k)$.
    Otherwise, if $j_1 = -1$, then $p_l \notin \mathcal{P}(A_l, b_l)$.

    For $t = 1$ holds that if  $p_1 > v^j_1$ and $j_1 = 1$, then there is an index $k \in \{1, \ldots, \NF\}$ with $p_{k, 1} > y^j_{k, 1}$ therefore $\textnormal{Proj}^1(p_{k}) \notin \mathcal{P}(A_k, b_k)$ and by Assumption \ref{assumption2} $p_k \notin \mathcal{P}(A_k, b_k)$.
    Otherwise, if $p_1 < v^j_1$ and $j_1 = -1$, then by similar reasoning $p_{k} \notin \mathcal{P}(A_k, b_k)$.
    Hence, all cases lead to contradictions and therefore $p \notin \mathcal{M}$.
\end{proof}

\subsection{Proof of Proposition 1}\label{sec:AppendixC}
\begin{proof}
	We prove the convex independence of the set of vectors $\{v^j : j \in \{-1, 1\}^\TP\}$ by induction over $\TP$. It then follows that $v^j$ is a vertex in $\mathcal{A}$.

	\noindent \textit{Base case}: ($\TP = 1$)
	In one-dimensional space two distinct numbers $v^{(0)}, v^{(1)}$ are computed cf. Lemma \ref{lemma:properties}, which are convex independent by definition.

	\noindent \textit{Induction hypothesis}:
	Let the set of vectors $\{v^j : j \in \{-1, 1\}^\TP\}$ be convex independent for a $\TP \in \mathbb{N}$.

	\noindent \textit{Induction step}: ($\TP \rightarrow \TP+1$) The $\TP+1$ dimensional vectors are constructed by $\{v^{(j, -1)}, v^{(j, 1)} : j \in \{-1, 1\}^\TP\}$. Assume the set of vectors $\{v^{(j, -1)}, v^{(j, 1)} : j \in \{-1, 1\}^\TP\}$ is convex dependent, then for some $k \in \{-1, 1\}^\TP$ we have without loss of generality that
	\begin{align}
		&v^{(k, -1)} = \sum_{j \in \{-1, 1\}^\TP, j \neq k} \alpha_{j}v^{(j, -1)} + \sum_{j \in  \{-1, 1\}^\TP}\beta_{j}v^{(j, 1)}\label{proof_conv_ind2}\\
		&\sum_{j \in \{-1, 1\}^\TP, j \neq k} \alpha_{j} + \sum_{j \{-1, 1\}^\TP}\beta_{j} = 1\\
		&\alpha_{j}, \beta_{j} \geq 0
	\end{align}
	Projecting \eqref{proof_conv_ind2} to the first $\TP$ coordinates gives:
	\begin{align*}
		&v^{k} = \beta_{k}v^k + \sum_{j \in \{-1, 1\}^\TP, j \neq k} (\alpha_{j} + \beta_{j})v^j\\
		&(1-\beta_{k})v^{k} = \sum_{j \in \{-1, 1\}^\TP, j \neq k} (\alpha_{j} + \beta_{j})v^j
	\end{align*}
	where $0 \leq \beta_{k} \leq 1$. We distinguish two cases:

	\noindent \textit{Case 1}: If $\beta_{k} < 1$, then $1- \beta_{k} > 0$ and we have:
	\begin{align*}
		v^{k} = \sum_{j \in \{-1, 1\}^\TP, j \neq k} \frac{(\alpha_{j} + \beta_{j})}{(1-\beta_{k})}v^j
	\end{align*}
	where $\frac{(\alpha_{j} + \beta_{j})}{(1-\beta_{k})} \geq 0$ and $\sum_{j \in \{-1, 1\}^\TP, j \neq k} \frac{(\alpha_{j} + \beta_{j})}{(1-\beta_{k})}=1$, hence $v^k$ is convex combination of vectors in $\{v^j : j \in \{-1, 1\}^\TP\} \setminus \{v^k\}$, which contradicts the induction hypothesis.

	\noindent \textit{Case 2}: If $\beta_{k} = 1$, then $\alpha_{j} = \beta_{j} = 0 \; \forall j \in \{-1, 1\}^\TP\setminus\{k\}$, and it follows from \eqref{proof_conv_ind2} that $v^{(k, -1)} = v^{(k, 1)}$, which is impossible as the vectors are distinct by Lemma \ref{lemma:properties}. 
    These contradictions show the convex independence of the vectors.

	\noindent Since $\mathcal{A} = \textnormal{Conv}(\{v^j : j \in \{-1, 1\}^\TP\})$, and $v^k$ for any $k \in \{-1,1\}^d$ is not a convex combination of vectors in $\{v^j : j \in \{-1, 1\}^\TP\} \setminus \{v^k\}$, it follows that $v^k$ is not a convex combination of vectors in $\mathcal{A} \setminus \{v^k\}$, which proofs that $v^k$ is a vertex of $\mathcal{A}$.
\end{proof}

\subsection{Proof of Proposition 2}\label{sec:AppnedixD}
\begin{proof}
	This statement is obvious if $\mathcal{M}\setminus \mathcal{A} = \emptyset$, i.e., $\mathcal{M}=\mathcal{A}$. Therefore, we temporarily suppose that $\mathcal{M}\setminus \mathcal{A} \neq \emptyset$ and we distinguish two cases:

	\noindent \textit{Case 1}: Let $p, q \in \mathcal{M} \setminus \mathcal{A}$. Assume that $v^j = tp+(1-t)q$ with $t \in (0, 1)$.
    Since $v^j\in \mathcal{A}$ and $p, q \not\in \mathcal{A}$, it follows that $p\neq v^j$ and $q \neq v^j$.
    Since $t \in (0, 1)$, it follows that $p \neq q$, hence there are indices in $\{1, \ldots, \TP\}$ where the entries in $p$ and $q$ are different.
    Let $m$ be the minimum of these indices.
    For this index holds $v^j_m = tq_m + (1- t)p_m$, $t \in (0, 1)$ and $p_t \neq q_t$. Without loss of generality, suppose that $p_m < q_m$, then $p_m < v^j_m < q_m$.
    Since $m$ is the minimum index, we have equality in $v, p$ and $q$ for the indices $\{1, \ldots, m-1\}$.
    If $j_m = 1$, then by Lemma \ref{lemma_not_in_M_sum} we have $q \not\in \mathcal{M}$. Otherwise, if $j_m = -1$, then  by Lemma \ref{lemma_not_in_M_sum} we see that $p \not\in \mathcal{M}$.
    Therefore we have a contradiction in both cases and the assumption must be false.
    Hence there are no $p, q \in \mathcal{M} \setminus \mathcal{A}$ with $v^j=tp + (1-t)q$ and $t \in (0, 1)$.

	\noindent \textit{Case 2:} Let $p \in \mathcal{M}\setminus \mathcal{A}$ and $q \in \mathcal{A}$.
    The proof for this case is almost a copy of the previous one. Assume that $v^j = tp + (1- t)q$ with $t$ $\in (0, 1)$.
    Since $p\not\in \mathcal{A}$, $q \in \mathcal{A}$ and $t \in (0, 1)$ it follows that $v^j\neq p$ and $p \neq q$.
    Since $p\neq q$, there is a minimum index $m$ where the components of $p$ and $q$ are different.
    For this index holds $v^j_m= tp_m+(1-t)q_m$ and $q_m \neq p_m$. Without loss of generality assume that $p_m < q_m$.
    Since $m$ is the minimum index, we have equality in the indices $\{1, \ldots, m-1\}$. If $j_m = 1$, then it follows by Lemma \ref{lemma_not_in_M_sum} that $q \not\in \mathcal{M}$ otherwise, if $j_m = -1$, then by the same reasoning it follows that $p\not\in \mathcal{M}$.
    This shows that there are no $p\in \mathcal{M} \setminus \mathcal{A}$ and $q\in \mathcal{A}$ with $v^j = tp + (1- t)q$ and $ t \in (0, 1)$.

	In conclusion, we see that the only possible case is $p, q\in \mathcal{A}$. This concludes the proof.
\end{proof}

\subsection{Proof of Lemma 3}\label{sec:AppendixE}
\begin{proof}
    If $y^j \in \mathcal{B}(S_0, S_f, p)$, then there is nothing to show. 
    Hence we assume that $y^j \notin \mathcal{B}(S_0, S_f, p)$. If the assignments in Lines 3 or 9 are applied, then
    \begin{align}\nonumber
        &\alpha^\TP S_0 + \sum_{\tau = 1}^\TP \alpha^{\TP-\tau}\Tilde{y}^j_{\tau}\Delta t = \alpha^\TP S_0 + \sum_{\tau = 1}^{\TP-1} \alpha^{\TP-\tau}\Tilde{y}^j_{\tau}\Delta t
        \\ & + \frac{S_{f} - (\alpha^\TP S_{0} + \sum_{\tau = 1}^{\TP-1} \alpha^{\TP-\tau}\Tilde{y}^j_{ \tau}\Delta t)}{\Delta t}\Delta t  = S_f.
    \end{align}
    Therefore, $\Tilde{y}^j \in \mathcal{B}(S_0, S_f, p)$ if there is a correction index $k \in \{1, \ldots, \TP\}$ such that one of the assignments are applied. If $k \in \{1, \ldots, \TP-1\}$ it holds that $y^j_{\tau} = \overline{x} \ \forall \tau \in \{k, \ldots, \TP-1\}$, and $k = \TP$ implies correction in Line 3 only. 
    
    Assume that $\frac{S_{f} - (\alpha^\TP S_{0} + \sum_{\tau = 1}^{\TP-1} \alpha^{\TP-\tau}\Tilde{y}^j_{\tau}\Delta t)}{\Delta t} > \overline{x}$ for all correction indices $k = 1, \ldots, \TP$, then also for $k=1$. We have $\overline{x} < \frac{S_{f} - (\alpha^\TP S_{0} + \sum_{\tau=1}^{\TP-1}\alpha^{\TP-\tau}\overline{x}\Delta t)}{\Delta t}$. Therefore, $S_{f} > \alpha^\TP S_{0} + \sum_{\tau = 1}^{\TP-1}\alpha^{\TP-\tau}\overline{x}\Delta t + \overline{x}\Delta t$. Since $\mathcal{B}(S_{0}, S_{f}, p) \neq \emptyset$, there exists an $x \in \mathcal{B}(S_{0}, S_{f}, p)$ with $S_{f} \leq \alpha^\TP S_{0} + \sum_{\tau = 1}^\TP\alpha^{\TP-\tau} x_{\tau}\Delta t$. This gives $\alpha^\TP S_{0} + \sum_{\tau =1}^{\TP}\alpha^{\TP-\tau}\overline{x}\Delta t < S_{f} \leq \alpha^\TP S_{0} + \sum_{\tau = 1}^\TP \alpha^{\TP-\tau}x_{\tau}\Delta t$ hence $\sum_{\tau=1}^\TP\alpha^{\TP-\tau}\overline{x} < \sum_{\tau = 1}^\TP\alpha^{\TP-\tau} x_{\tau}$, which implies that there is an index $m$ with $\overline{x} < x_m$ and therefore $x \notin \mathcal{B}(S_{0}, S_{f}, p)$, contradicting $x \in \mathcal{B}(S_{0}, S_{f}, p)$.
   
   From the above, we have that there are indices $k$ such that $\frac{S_{f} - (\alpha^\TP S_{0} + \sum_{\tau = 1}^{\TP-1}\alpha^{\TP-\tau} \Tilde{y}^j_{\tau}\Delta t)}{\Delta t} \leq  \overline{x}$. Hence we use the maximum correction index $l$ with $S_f \leq \alpha^\TP S_0 + \sum_{\tau = 1}^{l-1}\alpha^{\TP-\tau}\Tilde{y}^j_{\tau}\Delta t + \sum_{\tau=l}^{\TP} \alpha^{\TP-\tau}\overline{x}\Delta t$ and $S_f > \alpha^\TP S_0 + \sum_{\tau = 1}^{l}\alpha^{\TP-\tau} \Tilde{y}^j_{\tau}\Delta t + \sum_{\tau = l+1}^{\TP}\alpha^{\TP-\tau}\overline{x}\Delta t$.
    Note that this index exists since we assumed that $y^j \notin \mathcal{B}(S_0, S_f, p)$. Suppose that $\frac{S_{f} - (\alpha^\TP S_{0} + \sum_{\tau = 1}^{\TP-1} \alpha^{\TP-\tau}\Tilde{y}^j_{\tau}\Delta t)}{\Delta t} < \underline{x}$ for this index, then $S_f < \alpha^\TP S_0 + \sum_{\tau = 1}^{l-1}\alpha^{\TP-\tau} \Tilde{y}_{\tau}^j\Delta t+ \sum_{\tau=l}^{\TP-1}\alpha^{\TP-\tau}\overline{x}\Delta t + \underline{x}\Delta t$.
    This gives with $-S_f < -\alpha^\TP S_0 - \sum_{\tau = 1}^{l}\alpha^{\TP-\tau} \Tilde{y}^j_{\tau}\Delta t - \sum_{\tau = l+1}^{\TP}\alpha^{\TP-\tau}\overline{x}\Delta t$ that $0 < -\alpha^{\TP-l} \Tilde{y}^j_{l}\Delta t + \alpha^{\TP-l}\overline{x}\Delta t - \overline{x}\Delta t+ \underline{x}\Delta t$.
    Thus  $\alpha^{d-l}\Tilde{y}^j_l < (\alpha^{d-l}-1)\overline{x} + \underline{x} \leq \underline{x}$ and hence $\Tilde{y}^j_l < 0$.
    Since $\alpha^{\TP-l} \in (0, 1]$ we have that $\alpha^{\TP-l} \Tilde{y}_l^j \geq \Tilde{y}_l^j$. Therefore, $\Tilde{y}^j_{l} \leq \alpha^{\TP-l} \Tilde{y}_l^j < \underline{x}$, hence $\Tilde{y}^j_{l} < \underline{x}$, which is impossible as $y^j \in \mathcal{B}(S_0, \underline{S}, p)$ an all corrections were within $[ \underline{x}, \overline{x} ]$.
    We conclude that there exists an index such that the assignments in Line 3 or 9 are applied, and $\Tilde{y}^j \in \mathcal{B}(S_0, S_f, p)$.
\end{proof}

\subsection{Proof of Theorem 2}\label{sec:AppendixF}
\begin{proof}
    Assume that $\Tilde{y}^j$ is not a vertex of $\mathcal{B}(S_{0}, S_{f}, p)$, then there are $p, q \in \mathcal{B}(S_{0}, S_{f}, p)$ with $\Tilde{y}^j = pt + q(1-t)$, $p \neq q$ and $t \in (0, 1)$. We distinguish two cases.

    \noindent \textit{Case 1}: $\Tilde{y}^j = y^j$ and therefore $y^j$ is not changed by Algorithm~\ref{algo:positive correction batt}. Since $p, q \in \mathcal{B}(S_{0}, S_{f}, p)$, and $\mathcal{B}(S_{0}, S_{f}, p) \subseteq \mathcal{B}(S_{0}, \underline{S}, p)$ we have that $p,q \in \mathcal{B}(S_{0}, \underline{S}, p)$. Therefore we have $y^j = \Tilde{y}^j = pt + q(1-t)$, $p \neq q$ and $t \in (0, 1)$. Hence $y^j$ is not a vertex of $\mathcal{B}(S_{0}, \underline{S}, p)$, contradicting Theorem \ref{theorem:Vertices of M-sum}.

    \noindent \textit{Case 2}: $\Tilde{y}^j \neq y^j$ and therefore $y^j$ is changed by Algorithm~\ref{algo:positive correction batt}.
    Since $\Tilde{y}^j \neq y^j$, there exists a maximum correction index $f$ with $\Tilde{y}_{[f-1]}^j = y_{[f-1]}^j$ and $\Tilde{y}_{t}^j = \overline{x}$ $\forall t \in \{f, \ldots, d-1\}$.
    Since $p \neq q$, there is a minimum index $m$ with $p_{[m-1]} = p_{[m-1]}$ and $p_m \neq q_m$.
    Without loss of generality let $p_m > q_m$.
    If $f > m$, then $\Tilde{y}_{[m]}^j = y_{[m]}^j$, hence $p_{[m-1]} = y^j_{[m-1]} = q_{[m-1]}$ and $p_m > y_{m}^j > q_m$.
    From this it follows that $\textnormal{Proj}^m(p) \notin \mathcal{B}(S_{0}, \underline{S}, p)$ or $\textnormal{Proj}^m(q) \notin \mathcal{B}(S_{0}, \underline{S}, p)$. With Assumption~\ref{assumption2}  we have that $p \notin \mathcal{B}(S_{0}, \underline{S}, p)$ or $q \notin \mathcal{B}(S_{0}, \underline{S}, p)$, which contradicts $p, q \in \mathcal{B}(S_{0}, S_f, p)$.
    Hence, it holds that $f \leq m$. Moreover, $\Tilde{y}_t^j = \overline{x}$, $\forall t \in \{f, \ldots, d-1\}$ holds. For $m \neq d$, the inequality $\Tilde{y}_m^j < q_m$ implies that $q \notin \mathcal{B}(S_{0}, S_f, p)$.
    Hence $m = \TP$, which gives $q_{[\TP-1]} = p_{[\TP-1]} = \Tilde{y}^j_{[\TP-1]}$ and $p_\TP < \Tilde{y}^j_\TP < q_\TP$.
    This is, however, impossible as by Algorithm~\ref{algo:positive correction batt}  $\alpha^\TP S_{0, k} + \sum_{\tau = 1}^{\TP} \alpha^{\TP - \tau}\Tilde{y}^j_{k}\Delta t = S_{f, k}$. Using $p$ instead leads to $\alpha^\TP S_{0, k} + \sum_{\tau = 1}^{\TP} \alpha^{\TP - \tau}p_{k}\Delta t < S_{f, k}$.  Hence we conclude $p \notin \mathcal{B}(S_{0, k}, S_{f, k}, p_{k})$.
\end{proof}

\end{document}